\documentclass[a4paper,11pt]{article}
\usepackage{authblk}
\usepackage{mathrsfs}
\usepackage{graphics}
\usepackage{latexsym}
\usepackage{cite}
\usepackage{graphicx,psfrag}
\usepackage{amsmath,amssymb,amsthm,mathtools}
\usepackage[bf, small]{titlesec}
\usepackage{relsize}

\newtheorem{Thm}{Theorem}

\newtheorem{Prop}[Thm]{Proposition}
\newtheorem{Rem}[Thm]{Remark}
\newtheorem{Cor}[Thm]{Corollary}

\title{On consecutive edge magic total labeling of connected bipartite   graphs}
\author[1]{Bumtle Kang\thanks{This work was partially supported by
the National Research Foundation of Korea(NRF) grant funded by the Korea government(MEST) (No.\ NRF-2010-0009933).}}
\author[1]{Suh-Ryung Kim\textsuperscript{$\ast$}}
\author[2]{Ji Yeon Park\thanks{Corresponding author:  jypark0902@khu.ac.kr}}
\affil[1]{Department of Mathematics Education,
Seoul National University, Seoul 151-742, Korea} \affil[2]{Department of Mathematics, Kyung Hee University, Seoul 130-701, Korea}

\begin{document}
\maketitle
\begin{abstract}
Since Sedl\'{a}$\breve{\mbox{c}}$ek introduced the notion of magic labeling of a graph in 1963, a variety of magic labelings of a graph have been defined and studied. In this paper, we study consecutive edge magic labelings of a connected bipartite graph.  We make a very useful observation that there are only four possible values of $b$ for which a connected bipartite graph has a $b$-edge consecutive magic labeling. On the basis of this fundamental result, we deduce various interesting results on consecutive edge magic labelings of bipartite graphs, especially caterpillars and lobsters, which extends the results given by  Sugeng and Miller~\cite{SM}.
\end{abstract}


\section{Introduction}
Since Sedl\'{a}$\breve{\mbox{c}}$ek~\cite{se} introduced the notion of magic labeling of a graph, a variety of magic labelings of a graph have been defined and studied (see Gallian~\cite{Gal}). Kotzig and Rosa~\cite{KR}   introduced the notion of a magic valuation of a graph in 1970. A {\em magic valuation} of a graph $G=(V,E)$ is a bijection $f$ from $V \cup E$ to $\{1,2,\ldots, |V \cup E|\}$ such that for all edge $xy$, $f(x)+f(y)+f(xy)$ is a constant called the {\em magic constant} of $f$. Given a graph $G$, a positive integer $k$ is said to be a magic constant of $G$ if $k$ is the magic constant of a magic valuation of $G$.  Later, Ringel and Llad\'{o}~\cite{RL} rediscovered this notion and called it {\em edge-magic labeling}. More recently, Wallis~\cite{Wal} used the term {\em edge-magic total} labeling to distinguish it from other kinds of labeings that use the word magic.

An edge-magic total labeling is called a {\em $b$-edge consecutive magic labeling} if the edge labels are $\{b+1, b+2, \ldots, b+|E|\}$ where $0 \le b \le |V|$. Sugeng and Miller~\cite{SM} claimed that if a connected graph $G$ with $n$ vertices has a $b$-edge consecutive magic labeling with $1 \le b \le n-2$, then $G$ is a tree; the union of $r$ stars and a set of $r-1$ isolated vertices has an $r$-edge consecutive magic labeling.

In this paper, we extend these existing results about  consecutive edge magic labelings of graphs in the following way. We find all the values of $b$ for which a connected bipartite graph has a $b$-edge consecutive magic labeling; We show that a connected bipartite graph $G=(X,Y)$ having an $|X|$-edge consecutive magic labeling is a tree having a  graceful labeling and a super edge-magic labeling; We give a necessary and sufficient condition for a caterpillar having a $b$-edge consecutive magic labeling, which actually answers an open problem posed in Wallis~\cite{Wal}; We also obtain an interesting result on consecutive edge magic labelings for lobsters.

For any undefined term, the reader may refer to \cite{k10}.

\section{Properties of edge consecutive magic labelings of bipartite graphs}
We first present the following proposition which is simple but very useful.
\begin{Prop}
Suppose that $G$ has a $b$-edge consecutive magic labeling $\lambda$
for $1 \le b \le |V(G)|$. Then, if $y$ and $z$ are neighbors of $x$,
then either $\{\lambda(y),\lambda(z)\} \subset \{1,\ldots, b\}$ or $\{\lambda(y),\lambda(z)\} \subset \{b+|E(G)|+1, \ldots, |V(G)|+|E(G)|\}$.
\label{prop:neighbor}
\end{Prop}
\begin{proof}
Without loss of generality, we may assume that $\lambda(y) < \lambda (z)$. Suppose to the contrary that $\lambda(z) > b$ and $\lambda(y) < b+|E(G)|+1$. Then, by the definition of $b$-edge consecutive magic labeling, $\lambda(z) \ge b+|E(G)|+1$ and $\lambda(y) \le b$. Since $\lambda$ is an edge-magic total labeling,
\[\lambda(x)+\lambda(y)+\lambda(xy)=\lambda(x)+\lambda(z)+\lambda(xz),\]
which implies
\begin{equation}
\lambda(y)+\lambda(xy)=\lambda(z)+\lambda(xz)
\label{eq1}
\end{equation}
Since $\lambda$ is a $b$-edge consecutive magic labeling,
$\lambda(xz) \ge b+1$. From (\ref{eq1}) and the assumption that $\lambda(z) \ge b+|E(G)|+1$, we obtain
\[\lambda(y)+\lambda(xy) \ge (b+|E(G)|+1)+(b+1)=2b+|E(G)|+2.\]
Again, by the definition of $b$-edge consecutive magic labeling,
$\lambda(xy) \le b+|E(G)|$. Thus we have
\[\lambda(y) \ge 2b+|E(G)|+2-(b+|E(G)|)=b+2,\]
which contradicts the assumption that $\lambda(y) \le b$.
\end{proof}
The following theorem shows that there are only four possible  values of $b$ for
which a connected bipartite graph has a $b$-edge consecutive magic labeling.
\begin{Thm}
If a connected bipartite graph $G=(X,Y)$ has a $b$-edge consecutive magic labeling, then $b \in \{0,|X|,|Y|,|X|+|Y|\}$.
\label{thm:bipartite}
\end{Thm}
\begin{proof}
For simplicity, we denote the set $\{1, 2, \ldots, b\}$ by $[b]$. By the definition of a $b$-edge consecutive magic labeling, it is sufficient to show that $b=|X|$ or $|Y|$ if $1 \le b < |X|+|Y|$. Suppose that $1 \le b < |X|+|Y|$. We claim that either $\lambda(X) \subset [b]$ or $\lambda(X) \subset \{b+|E(G)|+1,\ldots, |X|+|Y|+|E(G)|\}$. Take two vertices $x$ and $x'$ in $X$.
Since $G$ is connected, there is an $(x,x')$-path $P$ in $G$. Since $G$ is bipartite,
\[P=x_1w_1x_2w_2x_3\cdots x_lw_lx'\]
where $x=x_1$ and $x_i \in X$ and $w_i \in Y$ for $i=1$, $\ldots$, $l$. Then, by Proposition~\ref{prop:neighbor},   if $\lambda(x) \le b$, then $\lambda(x_i) \le b$ for $i=2$, $\ldots$, $l$ and so $\lambda(x') \le b$, and  if $\lambda(x) \ge b+|E(G)|+1$, then  $\lambda(x_i) \ge b+|E(G)|+1$ for $i=2$, $\ldots$, $l$ and so $\lambda(x') \ge b+|E(G)|+1$.  Since $x$ and $x'$ were arbitrarily chosen, $\lambda(X) \subset [b]$ or $\lambda(X) \subset \{b+|E(G)|+1,\ldots, |X|+|Y|+|E(G)|\}$. Similarly we may show that  $\lambda(Y) \subset [b]$ or $\lambda(Y) \subset \{b+|E(G)|+1,\ldots, |X|+|Y|+|E(G)|\}$.

If $\lambda(X) \subset [b]$ and $\lambda(Y) \subset [b]$, then $b \ge |X|+|Y|$, which contradicts the assumption. If $\lambda(X) \subset \{b+|E(G)|+1,\ldots, |X|+|Y|+|E(G)|\}$ and $\lambda(Y) \subset \{b+|E(G)|+1,\ldots, |X|+|Y|+|E(G)|\}$, then $b=0$, which contradicts the assumption again. Therefore if $\lambda(X) \subset \{b+|E(G)|+1,\ldots, |X|+|Y|+|E(G)|\}$, then $\lambda(Y) \subset [b]$, or $\lambda(Y) \subset \{b+|E(G)|+1,\ldots, |X|+|Y|+|E(G)|\}$, then $\lambda(X) \subset [b]$.

If $\lambda(X) \subset [b]$ and $\lambda(Y) \subset \{b+|E(G)|+1,\ldots, |X|+|Y|+|E(G)|\}$, then $|X| \le b$ and $|Y| \le |X|+|Y|-b$, which implies $b=|X|$. If $\lambda(X) \subset \{b+|E(G)|+1,\ldots, |X|+|Y|+|E(G)|\}$ and $\lambda(Y) \subset [b]$, then $|X| \le |X|+|Y|-b$ and $|Y| \le b$, which implies $b=|Y|$.
\end{proof}
Suppose that a graph $G$ with $m$ edges has a labeling of its vertices with some subset of $\{0,1,\ldots,m\}$ such that no two vertices share a label and the edge labels are the set $\{1,2,\ldots,m\}$ where an edge label is the difference of the values assigned to its end vertices. Then $G$ is said to be {\em graceful} and such a labeling is called a {\em graceful labeling} of $G$.

 A $b$-edge consecutive magic labeling of a graph is called a {\it super edge-magic labeling} if $b=|V(G)|$. A graph $G$ is said to be {\it super edge-magic} if it has a  super edge-magic labeling.   Enomoto~{\it et al.}~\cite{ELNR} conjectured that every tree is super-edge magic, which still remains open.

We obtain an interesting result related to graceful labeling and super edge-magic labeling.  As a matter of fact, it gives a stronger version of the result by Sugeng and Miller~\cite{SM} stating that if a connected graph $G$ with $n$ vertices has a $b$-edge consecutive magic labeling with $1 \le b \le n-2$, then $G$ is a tree. Furthermore, the conjecture given by Enomoto~{\it et al.}~\cite{ELNR} is true for a tree $T$ that has an $|X|$-edge consecutive magic labeling where $X$ is one of bipartitions of $T$ when it is considered as a bipartite graph.

We first show the following.
\begin{Thm} \label{nonbipartite} If a connected non-bipartite graph $G$ has a
$b$-edge consecutive magic  labeling, then $b= 0$ or $|V(G)|$.
\end{Thm}
\begin{proof}
It suffices to show that if $b \ge 1$, then $b =|V(G)|$. Since $b \ge 1$, there is a vertex $x$ in $G$ such that $\lambda(x) \le b$. Take a vertex $y$ in $G$. Since $G$ is not a bipartite graph, there is an odd cycle $C$ in $G$. Let $w$ be a vertex on $C$. Since $G$ is connected, there is an $(x,w)$-path $P_1$ and a $(w,y)$-path $P_2$ in $G$.
Let $W$ be a walk which is obtained by concatenating $P_1$ and $P_2$ if the sum of lengths of $P_1$ and $P_2$ is even,  and by concatenating $P_1$, $C$, and $P_2$ if the sum of lengths of $P_1$ and $P_2$ is odd. In both cases, the walk $W$ is an $(x,y)$-walk of even length. Thus, by Proposition~\ref{prop:neighbor}, $\lambda(y) \le b$ since $\lambda(x) \le b$. Since $y$ was arbitrarily chosen, $\lambda(v) \le b$ for any $v \in V(G)$. Thus $|V(G)| \le b$, and therefore $|V(G)| = b$. \end{proof}
\noindent Let $G$ be a connected graph $G$ with $n$ vertices having a $b$-edge consecutive magic labeling for some $b \in [n-1]$. Then, by the above theorem, $G$ is bipartite. If $b=n-1$, then one of the partite sets of $G$ has size $n-1$ by Theorem~\ref{thm:bipartite} and so $G$ is a tree. Hence $G$ is a tree even for $b=n-1$.  The following theorem shows that $G$ should be a special tree that has both a graceful labeling and a super edge-magic labeling.
\begin{Thm}
If a connected bipartite graph $G=(X,Y)$ has an $|X|$-edge consecutive magic labeling, then $G$ is a tree having both a graceful labeling and a super edge-magic labeling.
\label{dual}
\end{Thm}
\begin{proof}
By the above observation, $G$ is a tree. We first show that $G$ has a graceful labeling.
Let $\lambda$ be an $|X|$-edge consecutive magic labeling of $G$. Then, by Proposition~\ref{prop:neighbor}, $\lambda(X)=\{1,\ldots, |X|\}$, or $|X|=|Y|$ and $\lambda(Y)=\{1,\ldots, |X|\}$. By symmetry, we may assume that $\lambda(X)=\{1,\ldots, |X|\}$. Then, by the definition of $|X|$-edge consecutive labeling, $\lambda(Y)=\{|X|+|E(G)|+1, \ldots, |X|+|Y|+|E(G)|\}$. Therefore
%
%
\begin{equation}
\lambda(X)=\{1,\ldots,|X|\}, \quad \lambda(Y)=\{|X|+|E(G)|+1,\ldots,|X|+|E(G)|+|Y|\},
\label{all}
\end{equation}
and
\[\lambda(E(G))=\{|X|+1, \ldots,|X|+|E(G)|\}.\]
Since $|\lambda(E(G))|=|E(G)|$ and $\lambda(x)+\lambda(y)+\lambda(xy)$ is constant for each edge $xy$ of $G$,
\begin{equation}
|\{\lambda(x)+\lambda(y) \mid x \in X, y \in Y, xy \in E(G)\}|=|E(G)|.
\label{edge}
\end{equation}
We define a labeling $\varphi$ from $X \cup Y$ to $\{0, 1,  \ldots, |X|+|Y|-1\}$ as follows:\[\varphi(z)=
\begin{cases}
\lambda(z) -1 & \mbox{if } z  \in X;
\\  |E(G)|+2|X|+|Y| - \lambda(z ) &  \mbox{if } z  \in Y.
\end{cases}
\]
\noindent Then, by (\ref{all}),
\begin{equation}
\varphi(X)=\{0,1,\ldots,|X|-1\} \quad\mbox{ and } \varphi(Y)=\{|X|, \ldots, |X|+|Y|-1\}
\label{sets}
\end{equation}
and
\[1 \le \varphi(y)-\varphi(x) \le |X|+|Y|-1.\]
Since $G$ is connected, $|X|+|Y|-1 \le |E(G)|$ and so
\begin{equation}
1 \le \varphi(y)-\varphi(x) \le |X|+|Y|-1 \le |E(G)|.
\label{absolute}
\end{equation}

Now, for any pair of adjacent vertices $x \in X$ and $y \in Y$, $\varphi(y) > \varphi(x)$ and
\begin{eqnarray}
\varphi(y) - \varphi(x)
 	                      &=& (|E(G)|+2|X|+|Y| - \lambda(y)) - (\lambda(x) -1) \notag \\
 	                      &=& |E(G)|+2|X|+|Y|+1 -(\lambda(x) + \lambda(y)). \label{eqs}
\end{eqnarray}
Thus, by (\ref{edge}) and (\ref{eqs}),
$S:= \{\varphi(y)-\varphi(x) \mid x \in X, y \in Y, xy \in E(G)\}$ has $|E(G)|$ elements.
Therefore, by (\ref{absolute}),
\[S=\{1, 2, \ldots, |E(G)|\}\]
which, together with (\ref{sets}), implies that   $\varphi$ is a graceful labeling of $G$.

Now we show that $G$ has a super edge-magic labeling. Let $\varphi$ be an $|X|$-edge consecutive magic labeling. Then $\varphi(X)=\{1, \ldots, |X|\}$; $\varphi(E(G)) = \{|X|+1, \ldots, |X|+|E(G)|\}$; $\varphi(Y)= \{|X|+|E(G)|+1,\ldots, |X|+|Y|+|E(G)|\}$.

Now we define $\varphi^*:X \cup Y \cup E(G) \to \{1, \ldots, |X|+|Y|+|E(G)|\}$ by
\[\varphi^*(v)=\begin{cases} \varphi(v) & \mbox{ if } v \in X, \\
\varphi(v)-|E(G)| & \mbox{ if } v \in Y,\end{cases}\]
and, for each pair of adjacent vertices $x$ and $y$,
\[\varphi^*(xy)=\varphi(xy)+|Y|.\]
For an edge $xy$ of $G$ for $x \in X$, $y \in Y$,
\[\varphi^*(x)+\varphi^*(y)+\varphi^*(xy)=\varphi(x)+(\varphi(y)-|E(G)|)+(\varphi(xy)+|Y|),\]
which is a constant number. Thus $\varphi^*$ is an edge-magic labeling of $G$. By the definition of $\varphi^*$, $\varphi^*$ is an $(|X|+|Y|)$-edge consecutive magic labeling, that is, a super edge-magic labeling.
\end{proof}


Given a graph $G$, let $\gamma:V(G) \cup E(G) \to \{1,2,\ldots,|V(G)|+|E(G)|\}$ be an edge-magic labeling for a graph $G$. Define the labeling $\gamma':V(G) \cup E(G) \to \{1,2,\ldots,|V(G)|+|E(G)|\}$ as follows:
For a vertex $x$,
\[\gamma'(x)  = |V(G)|+|E(G)|+1-\gamma(x),\]
for an edge $xy$,
\[\gamma'(xy)=|V(G)|+|E(G)|+1-\gamma(xy).\]
Then $\gamma'$ is called the {\em dual} of $\gamma$. From Wallis~\cite{Wal}, we know that the dual of an edge-magic labeling of a graph $G$ is also an edge-magic labeling of $G$. Moreover, if $k$ is the magic constant corresponding to $\gamma$, then for any adjacent vertices $x$ and $y$ of $G$,
\begin{align*}
&\gamma'(x)+\gamma'(y)+\gamma'(xy) \\ & =(|V(G)|+|E(G)|+1-\gamma(x))+(|V(G)|+|E(G)|+1-\gamma(y))+
(|V(G)|+|E(G)|+1-\gamma(xy)) \\
 & =3(|V(G)|+|E(G)|+1)-(\gamma(x)+\gamma(y)+\gamma(xy))=3(|V(G)|+|E(G)|+1)-k,
\end{align*}
that is, $3(|V(G)|+|E(G)|+1)-k$ is the magic constant corresponding to $\gamma'$.

From the fact that the dual of an edge-magic labeling of a graph $G$ is also an edge-magic labeling of $G$, the following theorem is immediately true.
\begin{Thm}
For a connected bipartite graph $G=(X,Y)$, exactly one of the following is true:
\begin{itemize}
\item[(i)] $G$ does not have a $b$-edge consecutive magic labeling for any $b$;
\item[(ii)] $G$ has only  $0$-edge consecutive magic labeling and  super edge-magic labeling;
\item[(iii)] $G$ is a tree having a $b$-edge consecutive magic labeling for each $b=0$, $|X|$, $|Y|$, $|X|+|Y|$.
\label{dual1}
\end{itemize}
\label{eachcase}
\end{Thm}
\begin{proof}
We suppose that $G$ has a $b$-edge consecutive magic labeling $\lambda$ and denote  the dual of $\lambda$ by $\lambda'$.  By Theorem~\ref{thm:bipartite}, $b=0$, $|X|$, $|Y|$, or $|X|+|Y|$.
Since $|V(G)|=|X|+|Y|$, it is true that
$\lambda(X \cup Y)=\{1,\ldots,|X|+|Y|\}$ and $\lambda(E(G))=\{|X|+|Y|+1, \ldots, |X|+|Y|+|E(G)|\}$
if and only if $\lambda'(X \cup Y)=\{|E(G)|+1,\ldots,|X|+|Y|+|E(G)|\}$ and $\lambda'(E(G))=\{1, \ldots, |E(G)|\}$.
Similarly,
$\lambda(X)=\{1,\ldots,|X|\}$, $\lambda(E(G))=\{|X|+1, \ldots, |X|+|E(G)|\}$, and $\lambda(Y)=\{|X|+1+|E(G)|,\ldots,|X|+|Y|+E(G)\}$
if and only if $\lambda'(Y)=\{1,\ldots,|Y|\}$, $\lambda'(E(G))=\{|Y|+1, \ldots, |Y|+|E(G)|\}$, and $\lambda'(X)=\{|Y|+1+|E(G)|,\ldots,|X|+|Y|+|E(G)|\}$.
Thus, for $b=|X|+|Y|$ or $|X|$, then the dual of a $b$-edge consecutive magic labeling of $G$ is a $(|V(G)|-b)$-edge consecutive magic labeling of $G$ and vice versa.
Since $\lambda$ is the dual of $\lambda'$, this statement is true even for $b=0$ or $|Y|$.  Therefore, if $b=|X|$ or $|Y|$, then, by the observation by Wallis~\cite{Wal} together with Theorems~\ref{dual}, the statement (iii) is true. If neither (i) nor (iii) is true, then $b=0$ or $|X|+|Y|$. By the above argument again, the statement (ii) is immediately true.
\end{proof}

Given a $b$-edge consecutive magic labeling $\lambda$ of a connected bipartite graph, there is a way of deducing a new  $b$-edge consecutive magic labeling and a magic constant from  $\lambda$ other than using  the dual of edge magic labeling.
\begin{Prop}
If $\lambda$ is a $b$-edge consecutive magic labeling of a connected bipartite graph $G=(X,Y)$ with a magic constant $k$, then the  mapping $\lambda^*:V(G) \cup E(G) \to \{1,\ldots,|V(G)|+|E(G)|\}$ defined by
\begin{itemize}
\item[(i)] $\lambda^*(x)=|V(G)|+2|E(G)|+1-\lambda(x)$ for a vertex $x$ and $\lambda^*(xy)=|E(G)|+1-\lambda(xy)$ for two adjacent vertex $x$ and $y$ is also a $b$-edge consecutive magic labeling of $G$ with magic constant $2|V(G)|+5|E(G)|+3-k$ if $b=0$;
\item[(ii)] $\lambda^*(x)=|X|+|Y|+1-\lambda(x)$ for a vertex $x$ and $\lambda^*(xy)=2|V(G)|+|E(G)|+1-\lambda(xy)$ for two adjacent vertex $x$ and $y$ is also a $b$-edge consecutive magic labeling of $G$  with magic constant $4|V(G)|+|E(G)|+3-k$ if $b=|V(G)|$;
\item[(iii)] $\lambda^*(x)=|X|+1-\lambda(x)$ for a vertex $x \in X$, $\lambda^*(y)=2|X|+|Y|+2|E(G)|+1-\lambda(y)$ for a vertex $y \in Y$,  and $\lambda^*(xy)=2|X|+|E(G)|+1-\lambda(xy)$ for two adjacent vertex $x$ and $y$ is also a $b$-edge consecutive magic labeling of $G$ with magic constant $5|X|+|Y|+3|E(G)|+3-k$ if $b=|X|$;
\item[(iv)] $\lambda^*(x)=|X|+2|Y|+2|E(G)|+1-\lambda(x)$ for a vertex $x \in X$, $\lambda^*(y)=|Y|+1-\lambda(y)$ for a vertex $y \in Y$,  and $\lambda^*(xy)=2|Y|+|E(G)|+1-\lambda(xy)$ for two adjacent vertex $x$ and $y$ is also a $b$-edge consecutive magic labeling of $G$ with magic constant $|X|+5|Y|+3|E(G)|+3-k$ if $b=|Y|$.
\end{itemize}

\label{lem:another}
\end{Prop}
\begin{proof} Suppose that $b=0$. Then
\begin{align*}
& \lambda^*(x)+\lambda^*(y)+\lambda^*(xy) \\  & \ \ =(|V(G)|+2|E(G)|+1-\lambda(x))+(|V(G)|+2|E(G)|+1-\lambda(y))+(|E(G)|+1-\lambda(xy))
\\ & \ \ =2|V(G)|+5|E(G)|+3-k,
\end{align*}
which is a constant number. For the remaining cases, it can similarly be checked.
\end{proof}

\section{Edge consecutive magic labelings for trees}
In the previous section, we have shown that if a connected bipartite graph $G$ has a $b$-consecutive labeling for some $b \in \{1,\ldots,|V(G)|-1\}$, then $G$ is a tree.
In this section, we study edge consecutive magic labelings of interesting families of trees.

A {\em caterpillar} is a tree derived from a path by joining leaves to the vertices of the path. We denote by $S_{n_1,n_2,\ldots,n_r}$ the caterpillar derived from a path $P_r=c_1c_2\cdots c_r$ for a positive integer $r$ by joining $n_i$ leaves to $c_i$, where $n_i$ is a nonnegative integer, for each $i=1$, $\ldots$, $r$. We denote the neighbors of $c_i$ by $c_{i,1}$, $\ldots$, $c_{i,n_i}$ for $i=1$, $\ldots$, $r$.
\begin{Thm}
The caterpillar $S_{n_1,n_2,\ldots,n_r}$ has a $b$-edge consecutive magic labeling if and only if \[t \in \{0,\sum_{i=1}^{\left\lceil\frac{r}{2}\right\rceil}n_{2i-1}+\left\lfloor\frac{r}{2}\right\rfloor,\sum_{i=1}^{\left\lfloor\frac{r}{2}\right\rfloor}n_{2i}+\left\lceil\frac{r}{2}\right\rceil,\sum_{i=1}^r n_i+r\}.\]

\label{thm:ctplr}
\end{Thm}
\begin{proof}
The caterpillar $S_{n_1,n_2,\ldots,n_r}$ is a connected bipartite graph. Let $(X,Y)$ be a bipartition of $S_{n_1,n_2,\ldots,n_r}$ such that
\begin{align*}
X&=\{c_1,c_{21},\ldots,c_{2n_2},c_3, \ldots, c_{r1},\ldots,c_{rn_r}\}; \\
Y&=\{c_{11}, \ldots, c_{1n_1},c_2,c_{31},\ldots,c_{3n_3} \ldots, c_{r-1,1},\ldots,c_{r-1,n_{r-1}},c_r\}
\end{align*}
 if $r$ is even, and
\begin{align*} X&=\{c_1,c_{21},\ldots,c_{2n_2},c_3, \ldots, c_{r-1,1},\ldots,c_{r-1,n_{r-1}},c_r\}; \\
Y&=\{c_{11}, \ldots, c_{1n_1},c_2,c_{31},\ldots,c_{3n_3} \ldots, c_{r1},\ldots,c_{rn_r}\}
\end{align*}
if $r$ is odd. Then $|X|=\sum_{i=1}^{\lfloor\frac{r}{2}\rfloor}n_{2i}+\lceil \frac{r}{2} \rceil$ and $|Y|=\sum_{i=1}^{\lceil\frac{r}{2}\rceil}n_{2i-1}+\lfloor\frac{r}{2} \rfloor$.

Now the `only if' part immediately follows from Theorem~\ref{thm:bipartite}.

For a notational convenience, we let  $\alpha=|X|=\sum_{i=1}^{\lfloor\frac{r}{2}\rfloor}n_{2i}+\lceil \frac{r}{2} \rceil$ and $\beta=|Y|=\sum_{i=1}^{\lceil\frac{r}{2}\rceil}n_{2i-1}+\lfloor\frac{r}{2} \rfloor$.
To show the `if' part, we give a a $\beta$-edge consecutive magic labeling first.
For each edge of $S_{n_1,n_2,\ldots,n_r}$, the labels of its end vertices are  $c_i$ and $c_{i+1}$ for some $i \in \{1, \ldots, r-1\}$, or  $c_{2i-1}$ and $c_{2i-1,j}$ for some $i \in \{1, \ldots, \lceil r/2 \rceil\}$ and $j \in \{1,\ldots, n_{2i-1}\}$, or $c_{2i}$ and $c_{2i,j}$ for some $i \in \{1, \ldots, \lceil r/2 \rceil\}$ and $j \in \{1,\ldots, n_{2i}\}$.

We define $\lambda:V(S_{n_1,n_2,\ldots,n_r})\cup E(S_{n_1,n_2,\ldots,n_r})\to\{1,2,\cdots,2\sum_{i=1}^rn_i + 2r-1\}$ by
\[\lambda(c_{2i-1})  =\alpha+2\beta-1+\sum_{l=1}^{i-1} n_{2l}+i;  \quad
\lambda(c_{2i})  =\sum_{l=1}^{i} n_{2l-1}+i\]   \[ \lambda(c_{2i-1,j})=\sum_{l=1}^{i-1} n_{2l-1}+i+j-1; \quad \lambda(c_{2i,j})=\alpha+2\beta-1+\sum_{l=1}^{i-1} n_{2l}+i+j\]
\[\lambda(c_{i}c_{i+1})  =\alpha+2\beta-i-\sum_{l=1}^{i}n_l \mbox{ for any } 1 \le i \le r-1;\]
\[\lambda(c_{2i-1}c_{2i-1,j})  =\alpha+2\beta-2i+2-\sum_{l=1}^{2i-2}n_l-j;\]
\[\lambda(c_{2i}c_{2i,j}) =\alpha+2\beta-2i+1-\sum_{l=1}^{2i-1}n_l-j\]
for $j=1$, $\ldots$, $n_{2i}$ and $i=1$, $\ldots$, $\left\lceil r/2 \right\rceil$.   Then
\begin{align*} \lambda(c_{2i-1})+\lambda(c_{2i})+\lambda(c_{2i-1}c_{2i})&=\left(\alpha+2\beta-1+\sum_{l=1}^{i-1} n_{2l}+i\right)+\left(\sum_{l=1}^{i} n_{2l-1}+i\right)\\ & \ \ \ \ +\left(\alpha+2\beta-2i+1-\sum_{l=1}^{2i-1}n_l\right) =2\alpha+4\beta.
\end{align*} Similarly, one can check that
\[\lambda(c_{2i-1})+\lambda(c_{2i-1,j})+\lambda(c_{2i-1}c_{2i-1,j})=\lambda(c_{2i})+\lambda(c_{2i,j})+\lambda(c_{2i}c_{2i,j})=2\alpha+4\beta.\]
 Thus $\lambda$ is a $\beta$-edge consecutive magic labeling. Hence the statement is true by Theorem~\ref{eachcase}.
\end{proof}

Since the double star $S_{m,n}$ is a special case of the caterpillar $S_{n_1,n_2,\ldots,n_r}$ for $r=2$, the following corollary is immediately true.
\begin{Cor}
The double star $S_{m,n}$ has a $b$-edge consecutive magic labeling if and only if $b \in \{0,m+1,n+1, m+n+2\}$.
\label{cor:doublestar}
\end{Cor}
\noindent Wallis~\cite{Wal} asked whether or not double stars are edge-magic. Since a consecutive edge magic labeling is an edge-magic labeling, the above proposition answers his question.

By Corollary~\ref{cor:doublestar}, the double star $S_{m,n}$ has an $(m+1)$-edge consecutive magic labeling and an $(n+1)$-edge consecutive magic labeling. We show that there are only two such labelings for each of $m$, $n$.
\begin{Prop}
For some positive integers $m$ and $n$, the double star $S_{m,n}$ has only two $(m+1)$-edge consecutive magic labelings (resp.\ $(n+1)$-edge consecutive magic labelings) both of which have magic constant $ 4m+2n+6$ (resp.\ $4n+2m+6$).
\label{thm:doublestarmagicnumber}
\end{Prop}
\begin{proof} We may regard $G:=S_{m,n}$ as a bipartite graph with bipartition $(X,Y)$ with $|X|=m+1$ and $|Y|=n+1$.
Let $\lambda$ be an $(m+1)$-edge consecutive magic labeling of $G$ and $k$ be the magic constant of $\lambda$. Then, by the definition of $(m+1)$-labeling and Proposition~\ref{prop:neighbor},
\begin{equation}
\lambda(X) =[m+1] \quad \mbox{ and } \quad \lambda(Y)=\{2m+n+3,\ldots,2m+2n+3\}
\label{eq:prop1}
\end{equation}
 for all $x \in X$ and $y \in Y$.
Let $u$ and $v$ be the central vertices of  $G$ and let $\lambda(u)=\alpha$ and $\lambda(v)=\beta$. Without loss of generality, we may assume that $u \in X$ and $v \in Y$.  As $\alpha$ and $\beta$ belong to different partite sets and the vertices other than the central vertices form an independent set, we know from (\ref{eq:prop1}) that the labeling $\lambda$ is completely determined by $\alpha$ and $\beta$. Therefore it suffices to show that there are only two possible pairs of integers for $(\alpha,\beta)$.
By (\ref{eq:prop1}) and the assumption that $u \in X$ and $v \in Y$,
\[A:=\{\lambda(u)+\lambda(y) \mid y \in Y \}=\{\alpha+2m+n+3, \cdots, \alpha+2m+2n+3\}\]
and\[
B: = \{\lambda(v)+\lambda(x) \mid x \in X \} = \{\beta+1, \cdots, \beta+m+1\}.\]
Since $X \cap Y =\emptyset$, it is true that $A \cap B=\{\lambda(u)+\lambda(v)\}$, so $|A \cap B|=1$. Moreover, $X \cup Y=V(G)$ and each edge is incident to $u$ or $v$,  so $A \cup B=\{ \lambda(x) + \lambda(y) \mid xy \in E(G)\}$.

Since $\lambda$ is an edge consecutive magic labeling, $\{k-\lambda(xy) \mid xy \in E(G)\}$ is a set of $m+n+1$ consecutive integers and therefore $\{ \lambda(x) + \lambda(y) \mid xy \in E(G)\}$ is a set of $m+n+1$ consecutive integers. Since $A \cup B=\{ \lambda(x) + \lambda(y) \mid xy \in E(G)\}$, $A \cup B$ is a set of $m+n+1$ consecutive integers. This together with the fact $|A \cap B|=1$ imply that there are only two possible cases:
\[\alpha+2m+n+3 = \beta+m+1 \quad \mbox{or} \quad \alpha+2m+2n+3 = \beta+1.\]
 Assume the former. Since $2m+n+3 \le \beta$ and $\alpha \le m+1$, \[
(2m+n+3)+m+1 \le \beta+m+1 =\alpha+2m+n+3 \le (m+1)+2m+n+3.\]
Since the left hand side of the first inequality and the right hand side of the second inequality both equal $3m+n+4$, we have  $\beta+m+1=3m+n+4$ and $\alpha+2m+n+3=3m+n+4$. Hence $\alpha=m+1$ and $\beta=2m+n+3$.  Now assume the latter. Since $\beta \le 2m+2n+3$ and $1 \le \alpha$, \[
1+2m+2n+3 \le \alpha +2m+2n+3 = \beta+1 \le 2m+2n+3+1.\]
Thus $\beta=2m+2n+3$ and $\alpha=1$. We can easily check that the magic constant is $4m+2n+6$ in both cases.

By symmetry, the double star $S_{m,n}$ has only two $(n+1)$-edge consecutive magic labelings and their magic constant is $2m+4n+6$. \end{proof}

Proposition~\ref{thm:doublestarmagicnumber} tells us that  magic constants of $(m+1)$-edge consecutive magic labelings and $(n+1)$-edge consecutive magic labelings for a double star are unique. As a matter of fact, the magic constants of a double star are of specific form.
\begin{Thm}
The magic constants of the double star $S_{m,n}$ are in the form of $dt+6$ for some nonnegative integer $t$ where $d$ is the greatest common divisor of $m$ and $n$.
\label{thm:magicconstant}
\end{Thm}
\begin{proof}
Suppose that the double star $G:=S_{m,n}$ has a $b$-edge consecutive magic labeling $\lambda$ and $k$ is a magic constant of $\lambda$. Let $x$ and $y$  be  the central vertices of the double star $G$ and $\lambda(x)=i$ and $\lambda(y)=j$. Then $1 \le i, j \le 2m+2n+3$.
\begin{align}
k(m+n+1)&=  \sum_{uv\in
E(G)}\left[\lambda(u)+\lambda(uv)+\lambda(v)\right] \notag
\\ &=  \frac{ [1+(2m+2n+3)](2m+2n+3)}{2}+m\lambda(x) +n\lambda(y) \notag
\\ & = (m+n+2)(2m+2n+3)+mi +nj \notag
\\ & = (m+n+1)(2m+2n+5)+mi +nj+1. \label{eq2}
\end{align}
Since magic constant $k$ is a positive integer,  $mi+nj+1$  is a multiple of $m+n+1$ by the equality (\ref{eq2}). That is, $mi+nj+1=l(m+n+1)$ for some positive integer $l$. Then, by (\ref{eq2}),
\begin{equation}
k=2m+2n+5+l
\label{sep}
\end{equation}
Since it is impossible for both $i$ and $j$ to equal $1$, we have $l \ge 2$.  Let $d$ be  the greatest common divisor of $m$ and $n$. Then $m=dm'$ and $n=dn'$ for relatively prime positive integers $m'$ and $n'$. Suppose that $d=1$, that is,  $m$ and $n$ are relatively prime. Then, since $l \ge 2$, by the B\'{e}zout's identity, $l-1=\mu_1m+\nu_1n$ or $l=\mu_1m+\nu_1n+1$ for some integers $\mu_1$ and $\nu_1$.  Then, by (\ref{sep}),
$k=2m+2n+\mu_1m+\nu_1n+6=d(2m+2n+\mu_1m+\nu_1n)+6$.

Now suppose $d \ge 2$. By the division algorithm, $l=dq+r$ for some integers $q$ and $r$ with $0 \le r \le d-1$.  Then
\[mi+nj+1=(m+n+1)(dq+r)\]
or
\[m(i-dq-r)+(j-dq-r)n-dq=r-1.\]
Since the left hand side is divisible by $d$, $r-1$ is a multiple of $d$. Since $r \le d-1$, $r=1$ and so $l=dq+1$. Hence, by (\ref{sep}),
\[k=2m+2n+5+(dq+1)=d(2m'+2n'+q)+6\]
and we complete the proof.
\end{proof}
We may regard $S_{m,n}$ as a bipartite graph with bipartition $(X,Y)$ with $|X|=n+1$ and $|Y|=m+1$. In the proof of Theorem~\ref{thm:ctplr}, we have shown that $2(n+1)+4(m+1)=4m+2n+6$ is a magic constant for $b=n+1$. For the $\beta$-edge consecutive magic labeling given in the proof of Theorem~\ref{thm:ctplr} where $\beta=|Y|$, we define $\varphi: V(S_{n_1,n_2,\ldots,n_r})\cup E(S_{n_1,n_2,\ldots,n_r})\to\{1,2,\cdots,2\sum_{i=1}^rn_i + 2r-1\}$ by
$\varphi(y)=\lambda(y)$ for each $y \in Y$; $\varphi(x)=\lambda(x)-\alpha-\beta+1$ for each $x \in X$;
$\varphi(xy)=\lambda(xy)+\alpha$
for each pair of adjacent vertices $x$ and $y$ of $S_{n_1,n_2,\ldots,n_r}$. It can easily be checked that $\varphi$ is super edge-magic labeling by the fact that $\lambda$ is a $\beta$-edge consecutive magic labeling. Now we take two adjacent vertices $x$ and $y$ of the caterpillar. Then we may assume that $x \in X$ and $y \in Y$. By the definition, \[\varphi(x)+\varphi(y)+\varphi(xy)=(\lambda(x)-\alpha-\beta+1)+\lambda(y)+(\lambda(xy)+\alpha)=2\alpha+3\beta+1=3m+2n+6.\] Then, by Lemma~\ref{lem:another}(ii),
\[4|V(G)|+|E(G)|+3-(3m+2n+6)=4(m+n+2)+(m+n+1)+3-(3m+2n+6)=2m+3n+6\]
 for $b=m+n+2$.
Furthermore, recalling that the dual $\gamma'$ of a $b$-edge consecutive magic labeling $\gamma$ with a magic constant $k$ is a $(|V(G)|-b)$-edge consecutive magic labeling with the magic constant $3(|V(G)|+|E(G)|+1)-k$, we obtain \[3(|V(G)|+|E(G)|+1)-(2m+3n+6)=4m+3n+6\] is a magic constant of $G$ for $b=0$. Then, by Lemma~\ref{lem:another}(i),
\[2|V(G)|+5|E(G)|+3-(4m+3n+6)=2(m+n+2)+5(m+n+1)+3-(4m+3n+6)=3m+4n+6\]
is another magic constant for $b=0$.
By the symmetry, $4(n+1)+2(m+1)=2m+4n+6$ is a magic constant for $b=m+1$.

In the rest of paper, we take a look at a special type of a lobster which is obtained from a star graph $G$ by attaching a leaf to each leaf of $G$. For a positive integer $p$, we denote by $L_p$ the lobster obtained from a star with $p$ leaves in such a way. In addition, we denote the center of $L_p$ by $x$, the vertices at distance $1$ from $x$ by $y_1$, $\ldots$, $y_p$, the vertex adjacent to $y_i$ by $x_i$ for each $i=1$, $\ldots$, $p$. Now the following is true for $L_p$.
\begin{Thm}
For $p \ge 3$, $L_p$ has a $b$-edge consecutive magic labeling if and only if $b \in \{0,2p+1\}$.
\label{thm:lobster}
\end{Thm}
\begin{proof}
Kim and Park~\cite{KP} showed that $L_p$ has an $(2p+1)$-edge consecutive magic labeling. Thus, by Theorem~\ref{dual1}, the `only if' part is true.

Now we show the `if' part. By Theorem~\ref{thm:bipartite}, $b \in \{0,p,p+1,2p+1\}$. For notational convenience, we denote $\{1,\ldots,p+1\}$ by $[p+1]$. To reach a contradiction, suppose that $b=p+1$. Then there is a $(p+1)$-edge consecutive magic labeling $\lambda$ of $L_p$ such that
\[\lambda(E(G))=\{p+2,\ldots, 3p+1\}.\]
By Proposition~\ref{prop:neighbor},
\[\lambda(\{x,x_1,\ldots,x_p\})=[p+1] \quad \mbox{and} \quad  \lambda(\{y_1,\ldots,y_p\})=\{3p+2, \ldots,4p+1\}.\]
Set $\lambda(x)=i$. Then $i \in [p+1]$.
Let $k$ be the magic constant corresponding to $\lambda$. Then the values of $\lambda$ for the edges joining $x$ and the vertices $y_1$, $\ldots$, $y_p$ are
\[k-4p-i-1, \ldots, k-3p-i-2.\]

Suppose that $i+j \in [p+1]$ for some integer $j$. Then $i+j$ is assigned to a vertex in $\{x,x_1,\ldots,x_p\}$. The set of  possible values of $\lambda$ for the edge joining a vertex in $\{y_1, \ldots, y_p\}$ and the vertex labeled with $i+j$ is
\[A(i,j):=\{k-[(i+j)+(4p+1)], \ldots, k-[(i+j)+(3p+2)]\}.\]
Then
\[A(i,1)=\{k-4p-i-2, \ldots, k-3p-i-3\}\]
and
\[A(i,-1):=\{k-4p-i, \ldots, k-3p-i-1\}.\]
We first note that all the elements of $A(i,1)$ except $k-4p-i-2=k-[(i+1)+(4p+1)]$ have been already assigned to edges joining $x$ and $y_1$, $\ldots$, $y_p$, and so  the vertex labeled with $4p+1$ must be joined with the one labeled with $i+1$. We also note that all the elements of $A(i,-1)$ except $k-3p-i-1=k-[(i-1)+(3p+2)]$ are occupied by edges joining $x$ and $y_1$, $\ldots$, $y_p$, and so  the vertex labeled with $3p+2$ must be joined with the one labeled with $i-1$

Now suppose that $i \ge 3$. Then
$i-2 \in [p+1]$ and $k-[(i-2)+(3p+2)]=k-3p-i$ is the only element in $A(i,-2)$ that was not assigned to edges joining $x$ and $y_1$, $\ldots$, $y_p$. Thus the vertex labeled with $i-2$ and the one labeled with $3p+2$ should be joined.  However,  $i-1 \in [p+1]$, and the vertex labeled with $3p+2$ must be joined to the one labeled with $i-1$ by the above argument and we reach a contradiction.

Now suppose that $i = 2$. Then $i+2=4 \in [p+1]$ since $p \ge 3$ and
\[A(i,2)=\{k-[(i+2)+(4p+1)], \ldots, k-[(i+2)+(3p+2)]\},\]
in which $k-[(i+2)+(4p+1)]=k-4p-i-3$ is the only available label. However, in that case, the vertex labeled with $i+2$ and the one labeled with $4p+1$ should be joined, which is a contradiction as the vertex labeled with $4p+1$ is already joined to the one labeled with $i+1$.

Now suppose that $i=1$. Then the values of $\lambda$ for the edges joining $x$ and vertices $y_1$, $\ldots$, $y_p$ are  $k-4p-2$, $\ldots$, $k-3p-3$.  Moreover,  $2=i+1 \in [p+1]$ since $p \ge 3$, and
\[A(1,1)=\{k-4p-3, \ldots, k-3p-4\},\]
in which $k-4p-3=k-[(1+1)+(4p+1)]$ is the only available label. Thus the vertex labeled $2$ and the one labeled with $4p+1$ are joined. Now $3=i+2 \in [p+1]$ since $p \ge 3$, and
\[A(1,2)=\{k-4p-4, \ldots, k-3p-5\}\] in which $k-4p-4=k-[(1+2)+(4p+1)]$ is the only available label for the edge incident to the vertex labeled with $3$. Then, however, the other end vertex of the edge must be labeled with $4p+1$, which is impossible as the vertex labeled with $4p+1$ is adjacent to the one labeled with $2$.

Suppose that $i=p+1$. Then the values of $\lambda$ for the edges joining $x$ and vertices $y_1$, $\ldots$, $y_p$ are  $k-5p-2$, $\ldots$, $k-4p-3$. Now for $p=i-1$,
 \[A(p+1,-1)=\{k-5p-1, \ldots, k-4p-2\}\]
in which $k-4p-2=k-[(p+1)+(-1)+(3p+2)]$ is the only available label for the edge incident to the vertex labeled with $p$. Thus the vertex labeled $p$ and the one labeled with $3p+2$ are joined. On the other hand, for $p-1=i-2$,
  \[A(p+1,-2)=\{k-5p, \ldots, k-4p-1\}\]
 in which $k-5p=k-[(p+1)+(-2)+(4p+1)]$ is the only available label for the edge incident to the vertex labeled with $p-1$. Then, however, the vertex labeled with $p-1$ and the one labeled $3p+2$ should be joined, which is impossible as the vertex labeled with $3p+2$ also must be adjacent to the one labeled with $p$.

Thus there is no $(p+1)$-edge consecutive magic labeling for $L_p$. Then, by Theorem~\ref{dual1}, the statement is true.
\end{proof}

\begin{Rem} For $p=1$ or $2$, the above theorem is false. We may assign $(p+1)$-edge consecutive magic labelings to $L_1$ and $L_2$, which are the paths $P_3$ and $P_5$ of lengths $2$ and $4$, respectively (see Figure~\ref{fig:counterexample}).
\end{Rem}
\begin{figure}
  \centering
  \psfrag{1}{$1$}\psfrag{1}{$1$}\psfrag{2}{$2$}\psfrag{3}{$3$}\psfrag{4}{$4$}
  \psfrag{5}{$5$}\psfrag{6}{$6$}\psfrag{7}{$7$}\psfrag{8}{$8$}
  \psfrag{9}{$9$}\psfrag{A}{$L_1$}\psfrag{B}{$L_2$}
\includegraphics[width=5cm]{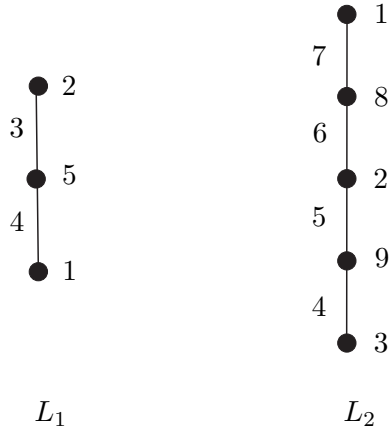}\\
  \caption{A $2$-edge consecutive magic labeling for $L_1$ and a $3$-edge consecutive magic labeling for $L_2$. }\label{fig:counterexample}
\end{figure}
\begin{Rem} By Theorem~\ref{thm:lobster}, we know that the converse of Theorem~\ref{dual} is false as the lobster $L_4$ is super edge-magic but it does not have a $b$-edge consecutive magic labeling for $b=4$ or $5$. Furthermore, $L_4$ is graceful as shown in Figure~\ref{fig:graceful}.
\end{Rem}
\begin{figure}
  \centering
\psfrag{0}{\small \bf 0}\psfrag{1}{\small \bf 1}\psfrag{2}{\small \bf 2}\psfrag{3}{\small \bf 3}\psfrag{4}{\small \bf 4}
  \psfrag{5}{\small \bf 5}\psfrag{6}{\small \bf 6}\psfrag{7}{\small \bf 7}\psfrag{8}{\small \bf 8}\psfrag{a}{\small\texttt 1}
  \psfrag{b}{\small\texttt 2}\psfrag{c}{\small\texttt 3}\psfrag{d}{\small\texttt 4}
  \psfrag{e}{\small\texttt 5}\psfrag{f}{\small\texttt 6}\psfrag{g}{\small\texttt 7}\psfrag{h}{\small\texttt 8}
\includegraphics[width=5cm]{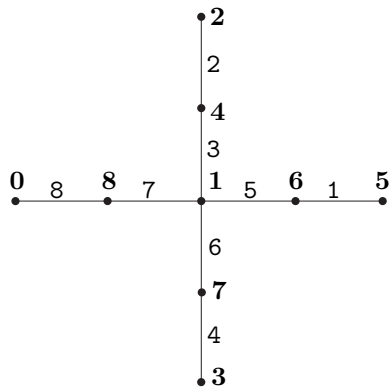}\\
  \caption{A graceful labeling of $L_4$. }
  \label{fig:graceful}
\end{figure}
\section{Closing Remarks}
We have shown that there are only two possible values of $b$ for which a connected non-bipartite graph has a $b$-edge consecutive magic labeling.

As Enomoto {\it et al.}~\cite{ELNR} showed that an odd cycle of length at least $3$ is super edge-magic and an even cycle does not have a super edge-magic labeling.  Thus, by Theorem~\ref{nonbipartite},  an odd cycle of length $l$ ($l \ge 3$) has a $b$-edge consecutive labeling if and only if $b \in \{0,l\}$.  Furthermore, by Theorems~\ref{dual} and \ref{eachcase}, an even cycle does have a consecutive edge magic labeling.

On the other hand, in the same paper, they proved that a complete bipartite graph $K_{m,n}$ is super edge-magic if and only if $m=1$ or $n=1$ or equivalently $K_{m,n}$ is not super edge-magic if and only if $m \ge 2$ and $n\ge 2$. This result together with Theorem~\ref{dual1} imply that $K_{m,n}$ has a total edge consecutive magic labeling  if and only if $m=1$ or $n=1$.

\end{document}